\documentclass{amsart}
\RequirePackage[utf8]{inputenc}
\usepackage{amsmath,amsthm,amsfonts,amssymb,amscd,amsbsy,multirow,hyperref}
\usepackage[all]{xy}
\usepackage{tikz}
\usepackage{graphicx,pgf,caption,subcaption}
\usepackage{enumerate}
\usepackage{dsfont}
\usepackage{stmaryrd}
\usepackage{ytableau}
\usepackage{amsmath,amsthm,amsfonts,amssymb,amscd,amsbsy,dsfont}
\usepackage{graphicx}
\usepackage{hyperref}

\newtheorem{corollary}{Corollary}
\newtheorem{definition}{Definition}
\newtheorem{lemma}{Lemma}
\newtheorem{proposition}{Proposition}

\newtheorem{theorem}{Theorem}

\numberwithin{equation}{section}

\allowdisplaybreaks

\title[ Charged Hawking Mass and Rigidity of minimal surfaces]{ A local rigidity theorem for minimal
two-spheres in an electrovacuum spacetime}

\author[H. Baltazar, A. Barros and R. Batista]{H. Baltazar, A. Barros and R. Batista}

\address[H. Baltazar] {Departamento de Matem\'{a}tica, Universidade Federal do Piau\'{\i}, 64049-550 Teresina, Piaui, Brazil.}
\email{halyson@ufpi.edu.br}

\address[A. Barros] { Departamento de Matem\'{a}tica,Universidade Federal do Cear\'{a} - UFC, Campus do Pici, Av. Humberto Monte, Bloco 914, Fortaleza, CE 60455-760, Brazil}
\email{abbarros@mat.ufc.br}

\address[R. Marcolino] {Departamento de Matem\'{a}tica, Universidade Federal do Piau\'{\i}, 64049-550 Teresina, Piaui, Brazil.}
\email{rmarcolino@ufpi.edu.br}

\subjclass[2000]{Primary 53A10; Secondary 53C24.}

\keywords{charged Hawking mass; scalar curvature; electrovacuum spacetime.}

\begin{document}

\newcommand{\spacing}[1]{\renewcommand{\baselinestretch}{#1}\large\normalsize}
\spacing{1.2}

\begin{abstract}
The purpose of this article is to prove that, under suitable constrains on the electrovacuum spacetime $M$, if $\Sigma\subset M$ is an embedded strictly stable minimal two-sphere which locally maximizes the charged Hawking mass, then there exist a neighborhood of it in $M$ isometric to the Reissner-Nordstr\"om-de Sitter space.  At the same time, motived by \cite{Brendle}, we will deduce an estimate for area of a two-sphere which is locally area minimizing in an electrovacuum spacetime. Moreover, if the equality holds, then there exist a neighborhood of it in $M$ isometric to the charged Nariai space.
\end{abstract}

\maketitle

\section{Introduction}\label{sectionInt}
A fascinating problem in differential geometry is to study Riemannian manifolds with scalar curvature lower bound.
In large part, this is because such manifolds have connections with general relativity. In the last decades, much efforts have been devoted to describe the topology and geometry of three-dimensional Riemannian manifolds $(M^3, g)$ for which the scalar curvature $R_g$ has a lower bound and under the assumption of the existence of a compact embedded surface in $M$ with some geometric property. Usually, stability as well as area minimizing are assumed.  Among the first celebrated works in this matter we mention those due to Schoen and Yau \cite{SY} and Meeks, Simon and Yau \cite{meeks}.

In the direction to understand this question Bray, Brendle and Neves \cite{Brendle} studied the structure of a three-manifold M under
the assumption of the existence of an embedded area minimizing two-sphere,
motived by a classical work due to Toponogov concerning closed geodesics
in surfaces of positive Gaussian curvature. It is important to observe  that
a rigidity statement as in \cite{Brendle} fails if we replace the condition of area minimizing by stability, due the counterexamples to Min-Oo’s conjecture \cite{coda}. A good survey due to Brendle \cite{brendle} gives an overview of this matter.

Recently, Maximo and Nunes \cite{mn} have obtained a similar rigidity
result by replacing the area minimizing assumption by strictly stability and
maximization of the Hawking Mass of a surface $\Sigma\subset(M^3, g)$.  More precisely, they have established a local rigidity result
assuming the existence of a locally maximizing two-sphere $\Sigma$ that is strictly stable in a three-dimensional manifold M with scalar curvature $R_g\geq 2$. Moreover, there exists a neighborhood of $M$ isometric to one of the de Sitter-Schwarzschild metric in $(-\epsilon,\epsilon)\times\Sigma)$.

On the other hand, one of the motivations to obtained rigidity results is the Positive Mass Theorem which states that an asymptotically flat three-manifold with nonnegative scalar curvature has nonnegative ADM mass (concept developed by Arnowitt, Deser and Misner in \cite{desser}) that was proved by Schoen and Yau in \cite{SY} and later by Witten \cite{witten} using spinors and the Dirac equation. Moreover, the ADM mass is zero if and only if $M$ is isometric to $\mathbb{R}^3$ with the standard flat metric.

In this paper, we are interested in analogue theorems proved by Maximo and Nunes in \cite{mn} for the Reissner-Nordstr${\rm \ddot{o}}$m-de Sitter  space in the charged setting as well as the results obtained by Bray, Brendle and Neves in \cite{Brendle} now  for charged Nariai space. These spaces are examples of the time-symmetric initial data set for the Einstein-Maxwell equation. In order to start describing such geometric models, remember the definition of initial data set for the Einstein-Maxwell equations.

Let $(\mathcal{M}^4, \gamma)$  a Lorentzian $4$-manifold and a $2$-form $F$ on $\mathcal{M}$ . The Einstein-Maxwell equations with cosmological constant $\Lambda\in\mathbb{R}$ consists of a triple $(\mathcal{M}^4, \gamma, F)$ satisfying the following system
{\setlength\arraycolsep{2pt}
\begin{eqnarray}\label{spacetime1}
Ric_{\gamma}-\frac{R_{\gamma}}{2}\gamma &+& \Lambda\gamma = 8\pi T_F, \\\label{spacetime2}
dF = 0;&&div_{\gamma} F = 0.
\end{eqnarray}}
Here $T_F$ denotes the \textit{electromagnetic energy-momentum tensor}
\begin{equation}\label{tensorelectromagnetic}
T_F =\frac{1}{4\pi}\Big(F\circ F-\frac{1}{4}|F|^2_{\gamma}\Big),
\end{equation}
where $F\circ F = \gamma^{\nu\mu}F_{\alpha\nu}F_{\beta\mu}$. The solutions of the \eqref{spacetime1}-\eqref{spacetime2} are called \textit{electrovacuum spacetimes}. The electromagnetic tensor $F$ admits a unique decomposition in terms
of the electric field $E$ and the magnetic field $B$, as measured by the static observer, for more details
see \cite{Gourgoulhon}.

A initial data set  $(M, g, K, E)$ for the Einstein–Maxwell equations with vanishing magnetic field is a spacelike hypersurface of $(\mathcal{M}^4,\gamma)$ with induced Riemannian metric $g$, a vector field $E$  and second fundamental form $K$ satisfying the \textit{Einstein-Maxwell constraint equations}
{\setlength\arraycolsep{2pt}
\begin{eqnarray}\label{initial data}
16\pi\mu &=& R_g + (Tr_gK)^2 -2\Lambda-|K|^2-2|E|^2\\
8\pi J & =& div(K - (Tr_gK)g),
\end{eqnarray}}where $\mu$ and $J$ are the energy and momentum densities of the matter fields.
In the time-symmetric case ($K=0$) our initial data sets represent a metric which is not changing with time.

The condition that the energy is positive for all observers gives rise to the charged dominant energy condition $\mu\geq |J|_g$. More precisely, for time-symmetric case is given by
\begin{definition}\label{dec} The time-symmetric initial data set $(M, g, E)$ is said to satisfy 
\begin{itemize}
\item[i)]The charged dominant energy condition if
$$R\geq 2\Lambda + 2|E|^2,$$
where $\Lambda\in\mathbb{R}$ plays the role of the cosmological constant.
\item[(i)]The Einstein-Maxwell constraint without charged matter if $div E =0$ everywhere on $M$.
\end{itemize}
\end{definition}

\begin{definition}\label{carga}
If $\Sigma\subset M$ is a closed orientable embedded surface, we define the charge $Q(\Sigma)$ relative to $E$ as
\begin{equation}\label{carga}
Q(\Sigma) = \frac{1}{4\pi}\int_{\Sigma}
\langle E,\nu\rangle d\sigma,
\end{equation}
where $\nu$ is a unit normal vector field on $\Sigma$. If $\Sigma$ bounds a volume, then $Q(\Sigma)$ is the total charge contained within $\Sigma$.
\end{definition}

Before to proceed let us provide some information for a special solution of the Einstein-Maxwell equations, such a structure is given by
\begin{equation}\label{m}
\gamma=-\Big(1-\frac{\Lambda r^2}{3}+\frac{Q^2}{r^2}-\frac{2m}{r}\Big)dt^2+\Big(1-\frac{\Lambda r^2}{3}+\frac{Q^2}{r^2}-\frac{2m}{r}\Big)^{-1}dr^2+r^2g_{\mathbb{S}^2},
\end{equation}
where $(\mathbb{S}^2, g_{\mathbb{S}^2})$ denotes the round sphere, $r$ is a radial parameter varying in a suitable open set $I\subset (0,\infty)$, and $t\in\mathbb{R}$. Is the so called\textit{ Reissner-Nordstr$\ddot{o}$m-de Sitter  spacetime} with mass parameter $m > 0$, electric charge $Q\in\mathbb{R}$ and positive cosmological constant $\Lambda\in\mathbb{R}$. The set $I$
depends on the solution of the following equation:
\begin{equation}\label{defrn}
\frac{\Lambda}{3}r^4-r^2 + 2mr- Q^2 = 0.
\end{equation}
\subsubsection*{$\bullet$ Reissner-Nordstr$\ddot{o}$m-de Sitter space}
The equation \eqref{defrn} has exactly three positive distinct and one negative real roots if and only if $0<Q^2<\frac{1}{4\Lambda}$ and
\begin{equation*}
\frac{2+\sqrt{1-4\Lambda Q^2}}{3\sqrt{2\Lambda}}\sqrt{1-\sqrt{1-4\Lambda Q^2}}<m<\frac{2-\sqrt{1-4\Lambda Q^2}}{3\sqrt{2\Lambda}}\sqrt{1+\sqrt{1-4\Lambda Q^2}}.
\end{equation*}
Its proof can be found in \cite[Proposition 1]{Mokdad}. If we denote by $r_{-}<r_+<r_c $ the positive roots of \eqref{defrn}.
The physical significance of these numbers is that $\{r = r_-\}$ is the inner
(Cauchy) black hole horizon, $\{r = r_+\}$ is the is the outer (Killing) black hole
horizon and $\{r = r_c\}$ is the cosmological horizon, and the smallest root has
no physical significance, since is negative.

For obtain a Riemannian metric we take $r\in (r_+, r_c)$, because the function $\Big(1-\frac{\Lambda }{3}r^2+\frac{Q^2}{r^2}-\frac{2m}{r}\Big)$ is positive in this region and negative in that one where $r \in (r_-, r_+)\cup (r_c,+\infty )$. Thus, the slice $\{t = 0\}$ of the Reissner-Nordstr${\rm \ddot{o}}$m-de Sitter
 spacetime defined on $(r_{+}, r_c)\times\mathbb{S}^2$ endowed with the metric 
\begin{equation}\label{Rnmetr}
g_{m,Q,\Lambda}=\Big(1-\frac{\Lambda}{3}r^2+\frac{Q^2}{r^2}-\frac{2m}{r}\Big)^{-1}dr^2+r^2g_{\mathbb{S}^2}
\end{equation}
is called \textit{Reissner-Nordstr\"om-de Sitter  space}.

To simplify the notation let us write  
\begin{equation}\label{simpl}
\rho(r)^2=(1-\frac{\Lambda}{3}r^2+\frac{Q^2}{r^2}-\frac{2m}{r}).
\end{equation}The electric field is given by $E = \frac{Q}{r^2}\rho(r)\partial r$ and scalar curvature is equal to $2|E|^2+2\Lambda$ everywhere. In a spherical slice, $\Sigma_r =\{r\}\times\mathbb{S}^2$, where $r\in(r_+, r_c)$, a unit normal is $N_r = \rho(r)\partial r$. So it is easy to
see that for all $r\in(r_+, r_c)$, the charge of the slice $\Sigma_r$ with respect to $N_r$ (as
defined in \eqref{carga}) is equal to $Q$.

By change of variable, the Reissner-Nordstr${\rm \ddot{o}}$m-de Sitter  metric can be rewritten as
\begin{equation}\label{metricc}
g_{m,Q,\Lambda} = ds^2 + u(s)^2g_{\mathbb{S}^2},\ \ \  {\rm on}\ \ \  [0, a]\times\mathbb{S}^2
\end{equation}
where $u: (0, a)\rightarrow (r_+, r_c)$ is a function that extends continuously to $[0, a]$ with $u(0) = r_+$, $u(a) = r_c$ and $\frac{ds}{dr}=\rho(r)^{-1}>0$ for $r\in(r_+,r_c)$.

After reflection of the metric $g_{m,Q,\Lambda}$, we can define a
complete periodic rotationally symmetric metric  on $\mathbb{R}\times\mathbb{S}^2$ with cosmological constant equal to $\Lambda$,
scalar curvature satisfying $R-2|E|^2=2\Lambda$ and charge $Q(\Sigma_r)=Q$ for each $r\in\mathbb{R}$. Moreover, the function $u$ solves the following second-order nonlinear
differential equation
\begin{equation}\label{edo}
u''(s)=\frac{1}{2}\Big(\frac{1-u'(s)^2}{u(s)}\Big)-\frac{1}{2}\Big(\frac{\Lambda u(s)^4+Q^2}{u(s)^3}\Big)
\end{equation}

\subsubsection*{$\bullet$ The charged Nariai space}

Now we assume that  there is a positive double root $0<r_{-}<r_{+}=r_{c}=\alpha$ and one root  negative of \eqref{defrn}, then
\begin{equation}
m = \alpha\Big(1-\frac{2}{3}\Lambda\alpha^2\Big)\ \ \ {
\rm and}\ \ \  Q^2 = \alpha^2 (1-\Lambda\alpha^2).
\end{equation}
Since the cosmological constant $\Lambda$ is positive, we have that $\alpha\in\Big(0,\frac{1}{\sqrt{\Lambda}}\Big)$ and 
$$r_{-} =\Big(\frac{3}{\Lambda}-2\alpha^2\Big)^{1/2}-\alpha.$$

If $\frac{1}{2\Lambda}<\alpha^2<\frac{1}{\Lambda}$ we obtain the \textit{charged Nariai space}. The function defined by \eqref{simpl} degenerates, since $r_+ =
r_c$, so, after a suitable coordinate transformation we can rewrite the metric \eqref{m}, such that,  each slice $\{t =
constant\}$ is a cylinder with a standard product metric, which extend smoothly
for all $s\in\mathbb{R}$. So, $\mathbb{R}\times\mathbb{S}$ endowed with the metric $g= ds^2 + \rho^2g_{\mathbb{S}^2}$ and $E(s) = \frac{Q}{\rho^2}\partial_s$ define the \textit{charged Nariai space}.

These spaces are examples of electrostatic systems, see Appendix for precise definition of this.
Standard reference for these spaces covered here can be found in \cite{brad}.

Now, following the procedure adopted in \cite{khuri}, we going to provide a new quasi-local mass adapted to initial data for the Einstein–Maxwell system which plays a crucial role in this paper.

Using \eqref{simpl}, we have that 
$$m(r) =\frac{1}{2}r\Big(1 -\frac{\Lambda}{3} r^2+ \frac{Q^2}{r^2}-\rho(r)^2\Big),$$
for every spherical slice, $\Sigma_r =\{r\}\times\mathbb{S}^2$ in the Reissner-Nordstr${\rm \ddot{o}}$m-de Sitter  space. By the spherical symmetry, every
spherical slice $\Sigma_r$ is umbilical and has constant mean curvature $H_r =\frac{2\rho(r)}{r}$ and $|\Sigma_r|=4\pi r^2$. 
Whence, we deduce that 
$$m(r) =\sqrt{\frac{|\Sigma_r|}{16\pi}}\Big(1 + \frac{4\pi Q^2}{|\Sigma_r|}-\frac{1}{16\pi}\int_{\Sigma_r}\Big(H^2+\frac{4}{3}\Lambda\Big)d\sigma_r\Big),$$
where we used that  $R-2|E|^2=2\Lambda$ and $E=\frac{Q}{r^2}\rho(r)\partial_r$.

\begin{definition} Given $(M,g)$ be a Riemannian three-manifolds  and a closed
$2$-surface $\Sigma\subset{M}$, the charged Hawking mass is defined to be
\begin{equation}\label{cmh}
m_{{\rm CH}}(\Sigma)=\Big(\frac{|\Sigma|}{16\pi}\Big)^{\frac{1}{2}}\Big(1-\frac{1}{16\pi}\int_{\Sigma}\Big(H^2+\frac{2}{3}\zeta\Big)d\sigma+\frac{4\pi Q(\Sigma)^2}{|\Sigma|}\Big),
\end{equation}
where $E\in\mathfrak{X}(M)$, $Q(\Sigma)=\frac{1}{4\pi}\int_{\Sigma}\langle E,\nu\rangle d\sigma$ is the total electric charge contained within $\Sigma,$  $H$ is the mean curvature of $\Sigma$ and $\zeta=\inf(R-2|E|^2)$.
\end{definition}
Note that $m_{CH}$ reduces to the ordinary Hawking mass when $Q(\Sigma) = 0$.

We remark that the slices in the Reissner-Nordstr${\rm \ddot{o}}$m-de Sitter  have constant charged Hawking mass. Indeed, it easy to check that for any slice $\Sigma_r$ of the $(\mathbb{R}\times\mathbb{S}^2, g_{m,Q,\Lambda})$ we have 
$$\frac{d}{dr}m_{{\rm CH}}(\Sigma_r)=\frac{1-u'(r)^2}{2}-\frac{1}{2}\Big(\frac{\Lambda u(r)^4+Q^2}{u(r)^2}\Big)-u(r)u''(r).$$
Since $u$ solves \eqref{edo} the  Reissner-Nordstr${\rm \ddot{o}}$m-de Sitter  space has constant scalar curvature
equal to $2|E|^2+2\Lambda$, which gives that each slice has constant charged Hawking mass.

Throughout the paper we will choose
the following normalization $\Lambda=1$. So that in particular the Reissner-Nordstr${\rm \ddot{o}}$m-de Sitter  metric has scalar curvature equal to $2|E|^2+2$.

Roughly speaking, our first result shows that slices in the Reissner-Nordstr${\rm \ddot{o}}$m-de Sitter space
are local maxima of the charged Hawking mass with respect to the embedded surfaces whose normal graphs have small enough $C^2$-norm. In this sense, inspired by ideas outlined in \cite{mn}, we will follow the same steps left in their article, now considering the Reissner-Nordstr${\rm \ddot{o}}$m-de Sitter space, in order to deduce our theorem. More precisely, we have the following result.

\begin{theorem}\label{gra}
Let $\Sigma_r=\{r\}\times\mathbb{S}^2$ be a slice of the Reissner-Nordstr$\ddot{o}$m-de Sitter manifold $(\mathbb{R}\times\mathbb{S}^2, g_{Q,a})$. Then there exists an $\epsilon = \epsilon(r) > 0$ such that if $\Sigma\subset \mathbb{R}\times\mathbb{S}^2$ is an embedded two-sphere, which is a normal graph over $\Sigma_r$ given by $\phi\in C^2(\Sigma_r)$ with $\|\phi\|_{C^2(\Sigma_r)} < \epsilon$, one has
\begin{itemize}
\item[i)] either $m_{{\rm CH}}(\Sigma) <  m_{{\rm CH}}(\Sigma_r)$;
\item[ii)] or  $\Sigma$ is a slice $\Sigma_s$ for some $s$.
\end{itemize}
\end{theorem}

In \cite{brad} Cruz et al. showed that every slice $\Sigma_r$ in the Reissner-Nordstr${\rm \ddot{o}}$m-de Sitter space satisfying
$$\frac{1-\sqrt{1-4Q^2}}{2}<r^2<\frac{1+\sqrt{1-4Q^2}}{2},$$
must be strictly stable (see next section
for a definition). In particular, $\Sigma_0$ is strictly stable. 

Recently, M\'aximo e Nunes in \cite{mn} were able to deduce a local rigidity for the de Sitter-Schwarzschild manifold, in this case the authors studied a strictly stable minimal two-sphere $\Sigma$ that locally minimizes the Hawking mass on a Riemannian three manifold $M$ and they proved that a neighborhood of it in $M$ must be isometric to the de Sitter-Schwarzshild metric $(-\epsilon,\epsilon)\times \Sigma$. Therefore, it is natural to ask what happens if we consider the same problem for a three manifold endowed of an electric field $E$ satisfying the initial data set for Einstein-Maxwell equation. In this case, we have established the following result.

\begin{theorem}\label{teoprincipal}
$(M, g, E)$ be a three-dimensional, time-symmetric initial data set for the Einstein-Maxwell equation with $E$ is the electric field. Assume that the charge density is zero ${\rm div} E = 0$, that the magnetic field vanishes, and that the non-electromagnetic matter fields satisfy the dominant energy condition $R\geq 2+2|E|^2$. If $\Sigma\subset M$ is an embedded strictly stable minimal two-sphere which locally maximizes the charged Hawking mass, then the Gauss curvature of $\Sigma$ is
constant equal to $1/a^2$ for some $a\in (0, 1)$ and a neighbourhood of $\Sigma$ in $(M,g)$
is isometric to the deSitter Reissner-Nordstr$\ddot{o}$m space $((-\epsilon, \epsilon) \times\Sigma, g_{Q,a})$ for some $\epsilon > 0$ and $Q=\frac{1}{4\pi}\int_{\Sigma}\langle E,\nu\rangle$.
\end{theorem}

To finish, we shall provide, in some sense, a generalization in the charged setting of the result due to Bray, Brendle and Neves \cite{Brendle}, see also \cite{Micallef}. We now state our last result.

\begin{theorem}\label{teoprincipalN}
Let $(M, g, E)$ be a three-dimensional, time-symmetric initial data set  for the Einstein-Maxwell equation with $E$ is the electric field. Assume that the charge density is zero $div E = 0$, that the magnetic field vanishes, and that the non-electromagnetic matter fields satisfy the dominant energy condition $R\geq 2 + 2|E|^2$. If $\Sigma$ is a 2-sphere in $M$ which
is locally area minimizing, then the area of $\Sigma$ satisfies,
\begin{equation}\label{inqN}
|\Sigma|+\frac{16\pi^2Q(\Sigma)^2}{|\Sigma|}\leq 4\pi.
\end{equation}
Moreover, if equality holds then the Gauss curvature of $\Sigma$ is constant equal to $1+|E|^2$ and a neighbourhood of $\Sigma$ in $(M, g)$  is isometric to the charged Nariai.
\end{theorem}

\section{Preliminaries}
In this section we shall present a couple of results that will be useful in the proof of our main theorem. Let $\Sigma\subset(M,g)$ be a surface with unit normal vector field $\nu,$ second fundamental form $A,$ and mean curvature $H$. The Jacobi Operator associated to $\Sigma$ is given by
\begin{equation}\label{jacobi}
 L = \Delta_{\Sigma}+Ric(\nu, \nu) + |A|^2,
\end{equation}
where $\Delta_{\Sigma}$ denotes the Laplacian of $\Sigma$ in the induced metric and $Ric$ is the Ricci curvature of $M$. As usual, we denote by $\lambda_1(L)$  the first eigenvalue  of $L$.

Hence, we may associate to $L$ a quadratic form $J(\phi)=-\int_{\Sigma}\phi L\phi d\sigma$, for $\phi\in C^{\infty}(\Sigma)$.  An embedded surface $\Sigma$ is called stable if and only if $J(\phi)\geq 0$ for all $\phi\in C^{\infty}(\Sigma)$, in this case $\lambda_1(L)$ is nonnegative. If the  smallest eigenvalue is positive we say the surface is \textit{strictly stable}. Equivalently, $J(\phi)\geq \lambda_1(L)\int_{\Sigma}\phi^2d\sigma$, which becomes
\begin{equation}\label{stability}
\lambda_1(L)\int_{\Sigma}\phi^2d\sigma+\int_{\Sigma}(Ric(\nu,\nu)+|A|^2)\phi^2d\sigma\leq\int_{\Sigma}|\nabla \phi|^2d\sigma,
\end{equation}
for any $\phi\in C^{\infty}(\Sigma)$.

Proceeding, we should consider $\Sigma_{t}\subset M$ as a smooth normal variation of $\Sigma$, i.e.,  
$$\Sigma_{t}=\{f(t,x):x\in\Sigma\}$$
where $f:(-\epsilon,\epsilon)\times\Sigma\rightarrow M$ is a smooth map such that: 
\begin{itemize}
\item for each $t\in(-\epsilon,\epsilon),$ $f_{t}=f(t,\cdot)$ is an immersion;
\item $f(\{0\}\times\Sigma)=\Sigma$;
\item $\frac{\partial f}{\partial t}=\phi(x)\nu(x)$ for each $x\in\Sigma,$ where $\phi\in C^{\infty}(\Sigma).$
\end{itemize}

Before presenting the first variation formula of the charged Hawking mass, let us point out that throughout this paper we will be considering the charge density zero, i.e., ${\rm div}=0$ and consequently, using the Divergence Theorem over the region enclosed between $\Sigma_{0}$ and $\Sigma(t)$, we have   
\begin{eqnarray*}
0&=&\frac{1}{4\pi}\Big(\int_{\Sigma(t)}\langle E, \nu_t\rangle d\sigma_t+\int_{\Sigma_0}\langle E, \nu_0\rangle d\sigma\Big)\\
&=&Q(\Sigma(t))-Q(\Sigma_0),
\end{eqnarray*}
which implies that $Q(\Sigma(t))=Q(\Sigma_0)$ for all smooth normal variations $\Sigma(t)$ of $\Sigma$. Now, after to apply \cite[Proposition A.1]{mn} we obtain:

\begin{proposition}[First variation of the charged Hawking mass]\label{var1} Under the considerations of (\ref{cmh}) we have{\setlength\arraycolsep{2pt}
\begin{eqnarray*}
\frac{d}{dt}m_{{\rm CH}}(\Sigma_{t})\Big|_{t=0}&=&- \frac{2|\Sigma|^{1/2}}{(16\pi)^{3/2}}\int_{\Sigma}\phi\Delta_{\Sigma} H d\sigma+\frac{|\Sigma|^{1/2}}{(16\pi)^{3/2}}\int_{\Sigma}(\Lambda-R)H\phi d\sigma\\
&+&\frac{|\Sigma|^{1/2}}{(16\pi)^{3/2}}\int_{\Sigma}\Big[2K_{\Sigma}-\frac{8\pi}{|\Sigma|}+\Big(\frac{32\pi^2 Q(\Sigma)^2}{|\Sigma|^2}-|A|^2\\
&+&\frac{1}{2|\Sigma|}\int_{\Sigma}H^2d\sigma\Big)\Big]H\phi d\sigma.
\end{eqnarray*}}
\end{proposition}

For our purpose, considering a minimal surface, the second variation formula can be expressed as follows: 

\begin{proposition}\label{2vari}(Second variation of the charged Hawking mass). If $\Sigma\subset M$ is a minimal surface that is critical point of the charged Hawking mass, then
{\setlength\arraycolsep{2pt}
\begin{eqnarray*}
\frac{d^2}{dt^2}m_{{\rm CH}}(\Sigma(t))\Big|_{t=0}&=&\frac{|\Sigma|^{1/2}}{32\pi^{3/2}}\left[\left(\frac{|\Sigma|\Lambda-8\pi}{2|\Sigma|}+\frac{16\pi^{2}}{|\Sigma|^{2}}Q(\Sigma)^{2}\right)\int\phi L\phi d\sigma-\int_{\Sigma}(L\phi)^{2}d\sigma\right]
\end{eqnarray*}}
\end{proposition}
\begin{proof}
Since the proof of this lemma is very short, we include it here for the sake of completeness. To start, taking into account the identities below
\begin{enumerate}
\item[(i)]  $\displaystyle\frac{d}{dt}(d\sigma_{t})\Big|_{t=0}=-\phi Hd\sigma$
\item[(ii)] $\displaystyle\frac{d^{2}}{dt^{2}}(d\sigma_{t})\Big|_{t=0}=[|\nabla\phi|^{2}-(Ric(\nu,\nu)+|A|^{2}){\phi}^{2}+H^{2}\phi^{2}+{\rm div}_{\Sigma}(\nabla_{X}X)]d\sigma,$
\end{enumerate}
where $X(x)=\frac{\partial f}{\partial t}(0,x),$ and as we are considering the minimal case, it is immediate to deduce

\begin{eqnarray*}
\frac{d^{2}}{dt^{2}}m_{CH}(\Sigma(t))\Big|_{t=0}&=&\frac{1}{8\pi^{1/2}|\Sigma|^{1/2}}\frac{d^{2}}{dt^{2}}(d\sigma_{t})\Big|_{t=0}-\frac{|\Sigma|^{1/2}}{(16\pi)^{3/2}}\frac{d^{2}}{dt^{2}}\left(\int_{\Sigma_{t}}H_{t}^{2}d\sigma_{t}\right)\Big|_{t=0}\\
&&-\frac{\Lambda|\Sigma|^{1/2}}{(16\pi)^{3/2}}\frac{d^{2}}{dt^{2}}(d\sigma_{t})\Big|_{t=0}-\frac{\pi^{1/2}Q^{2}}{2|\Sigma|^{3/2}}\frac{d^{2}}{dt^{2}}(d\sigma_{t})\Big|_{t=0}\\
&=&-\frac{1}{8\pi^{1/2}|\Sigma|^{1/2}}\int_{\Sigma}\phi L\phi d\sigma-\frac{|\Sigma|^{1/2}}{32\pi^{3/2}}\int_{\Sigma}(L\phi)^{2}d\sigma\\
&&+\frac{\Lambda|\Sigma|^{1/2}}{64\pi^{3/2}}\int_{\Sigma}\phi L\phi d\sigma+\frac{\pi^{1/2}Q^{2}}{2|\Sigma|^{3/2}}\int_{\Sigma}\phi L\phi d\sigma.
\end{eqnarray*}
This gives the requested result.
\end{proof}

\section{Charged Hawking mass and rigidy results}
In this section we shall prove Theorem~\ref{gra} announced in Section~\ref{sectionInt}. To do so, we start establishing some properties of a compact two-sided surface immersed in a Riemannian manifold endowed of a vector field $E$ satisfying ${\rm div}E=0$ and scalar curvature bounded from below by $2+2|E|^{2}.$ More precisely, we have the following result.

\begin{lemma}\label{cricalpoint}
Let $(M^3, g)$ be a Riemannian manifold, $E\in\mathfrak{X}(M)$ such that $div E=0$ with scalar curvature $R-2|E|^2\geq2$ and let us consider $\Sigma\subset M$ a compact two-sided surface with nonnegative mean curvature. In addition, if it is a critical point
for the charged Hawking mass, then
\begin{itemize}
\item[i)] either $\Sigma$ is minimal;
\item[ii)] or $\Sigma$  is umbilic, its Gaussian curvature satisfies $K_{\Sigma} = \frac{4\pi}{|\Sigma|}$ and  along $\Sigma$ $R-2|E|^2=2$, $E$ and $\nu$ normal field on $\Sigma$ are linearly dependent and $|E|^ 2$ is constant.
\end{itemize}
In particular, a closed two-sided surface $\Sigma$ in the deSitter Reissner-Nordstr$\ddot{o}$m $(\mathbb{R}\times\mathbb{S}^2, g_{Q,a})$ with nonnegative mean curvature is a critical
point of the charged Hawking mass if and only if it is minimal or a slice.
\end{lemma}
\begin{proof}
Firstly, note that we may write the first variation of the charged Hawking mass presented
in Proposition \ref{var1}, as
$$
\frac{d}{dt}m_{{\rm CH}}(\Sigma(t))\Big|_{t=0}=- \frac{2|\Sigma|^{1/2}}{(16\pi)^{3/2}}\int_{\Sigma}(\Delta_{\Sigma}H+\mathcal{Z}(\Sigma)H)\phi d\sigma,$$
where $\mathcal{Z}(\Sigma)$ is the following function defined on any properly embedded surface
\begin{equation}\label{eq1}
\mathcal{Z}(\Sigma) =\frac{4\pi}{|\Sigma|} -K_{\Sigma}  -\frac{16\pi^2Q(\Sigma)^2}{|\Sigma|^2}+\frac{1}{2}(R-\zeta) +
\frac{1}{4}\Big(2|A|^2- \frac{1}{|\Sigma|}\int_{\Sigma}H^2 d\sigma\Big).
\end{equation}
Since $R\geq 2|E|^2+2$ and $|A|^2\geq \frac{1}{2}H^2,$ we use Gauss-Bonnet formula and the H${\rm \ddot{o}}$lder inequality, to arrive at
\begin{equation}\label{ld}
\int_{\Sigma}\mathcal{Z}(\Sigma)d\sigma\geq- \frac{16\pi^2Q(\Sigma)^2}{|\Sigma|}+\int_{\Sigma}|E|^2d\sigma\geq 0.
\end{equation}
Thus, as $\Sigma$ is a critical point we conclude
\begin{equation}
\int_{\Sigma}(\Delta_{\Sigma}H+\mathcal{Z}(\Sigma)H)\phi d\sigma=0. 
\end{equation}
By our hypothesis, $H\geq 0$. Then,
we may apply the Maximum Principle to infer that either $H\equiv 0$ or $H >0$. Supposing $\Sigma$ not minimal we have
\begin{equation}\label{eq2}
\frac{1}{H}\Delta_{\Sigma}H+\mathcal{Z}(\Sigma)=0.
\end{equation}
Therefore, upon integrating the identity \eqref{eq2} over $\Sigma$  we infer
\begin{equation*}
\int_{\Sigma}\mathcal{Z}(\Sigma)d\sigma=-\int_{\Sigma}\frac{|\nabla H|^2}{H^2}d\sigma.
\end{equation*}
This allow us to conclude that $\int_{\Sigma}\mathcal{Z}(\Sigma)d\sigma = 0$. From where we conclude that $H$ is constant. Now, from \eqref{eq2} we deduce that $\mathcal{Z}(\Sigma) = 0$ which implies that $\Sigma$ is umbilic, $R = 2|E|^2+ 2$ on $\Sigma$ and $K_{\Sigma}=\frac{4\pi}{|\Sigma|}$. Moreover, the equality in \eqref{ld}, implies  $E$ and $\nu$ normal field on $\Sigma$ are linearly dependent and $|E|^ 2$ is constant along $\Sigma$.
\end{proof}

According to Lemma \ref{cricalpoint}, $\Sigma_r$ is critical point for the charged Hawking mass,
moreover $\Sigma_r$ is umbilic, its Gaussian curvature is constant and the scalar curvature of $M^3$ restrict to $\Sigma_r$ satisfies $R=2|E|^2+2$. Now, since $|A_r|^2=\frac{H_r^2}{2}$ and
$$Ric(\nu_r,\nu_r)=\frac{1}{2}R-K_{\Sigma_r}+\frac{1}{2}H_r^2-\frac{1}{2}|A_r|^2,$$ 
we may deduce that  
\begin{equation}\label{Ricci}
Ric(\nu_r,\nu_r)+\frac{H_r^2}{2}=\frac{8\pi}{|\Sigma_r|}+\frac{64\pi^2Q^2}{|\Sigma_r|^2}-\frac{3}{4}\Big(\frac{16\pi}{|\Sigma_r|}\Big)^{3/2}m_{{\rm CH}}(\Sigma_r).
\end{equation}

Moreover, considering the minimal case, we use (\ref{cmh}) in order to rewrite (\ref{Ricci}) as 
\begin{equation}\label{RicH}
Ric(\nu_r,\nu_r)=1-\frac{4\pi}{|\Sigma_r|}+\frac{16\pi^{2}Q^{2}}{|\Sigma_r|^{2}}.
\end{equation}

\begin{proposition} Let $(\mathbb{R}\times\mathbb{S}^ 2 , g_{Q, a} )$ be the de Sitter Reissner-Nordstr$ \ddot{o}$m
with mass $m_{Q}>0$ and let $\Sigma_{r} = \{r\}\times\mathbb{S}^2$ be a minimal slice with $r\in\{2n\ ;\ n\in\mathbb{N}\}$. Then, there exists a constant $C =C(\Sigma_{r} ) > 0$ such that for all smooth normal variation $\Sigma(t)$ of $\Sigma_{r}$
\begin{equation}\label{des}
\frac{d^2}{dt^2}m_{{\rm CH}}(\Sigma(t))\Big|_{t=0} \leq  -C\int_{\Sigma_0}(\phi-\overline{\phi})^2d\sigma_0
\end{equation}
\end{proposition}
\begin{proof}
Firstly, we substitute (\ref{RicH}) in Proposition~\ref{2vari}, to obtain
\begin{eqnarray*}
\frac{d^2}{dt^2}m_{{\rm CH}}(\Sigma(t))\Big|_{t=0}&=&\frac{|\Sigma_r|^{1/2}}{32\pi^{3/2}}Ric(\nu_r,\nu_r)\int_{\Sigma_r}(\phi\Delta\phi+Ric(\nu_r,\nu_r)\phi^{2})d\sigma_{r}\\
&&-\frac{|\Sigma_r|^{1/2}}{32\pi^{3/2}}\int_{\Sigma_r}(\Delta\phi+Ric(\nu_r,\nu_r)\phi)^{2}d\sigma_{r}\\
&=&\frac{|\Sigma_r|^{1/2}}{32\pi^{3/2}}\left[Ric(\nu_r,\nu_r)\int_{\Sigma_r}|\nabla\phi|^{2}d\sigma_r-\int_{\Sigma_r}(\Delta\phi)^{2}d\sigma_r\right],
\end{eqnarray*}
which can be rewritten for our purpose as
\begin{eqnarray*}
\frac{d^2}{dt^2}m_{{\rm CH}}(\Sigma(t))\Big|_{t=0}&=&-\frac{|\Sigma_r|^{1/2}}{32\pi^{3/2}}\left(\int_{\Sigma_r}(\Delta\phi)^{2}d\sigma_r-\int_{\Sigma_r}|\nabla\phi|^{2}d\sigma_r\right)\nonumber\\
&&+\frac{1}{96\pi^{3/2}|\Sigma_r|^{1/2}}\left(-12\pi+\frac{48\pi^{2}Q^{2}}{|\Sigma_r|}\right)\int_{\Sigma_r}|\nabla\phi|^{2}d\sigma_r.
\end{eqnarray*} 

On the other hand, from Corollary~\ref{inq1B} of the Appendix, the area of $\Sigma_r$ satisfies
\begin{equation}\label{BB}
|\Sigma_r|+\frac{48\pi^2Q^2}{|\Sigma_r|}\leq 12\pi.
\end{equation}

Now, we can use \eqref{BB} to get
{\setlength\arraycolsep{2pt}
\begin{eqnarray}\label{auxACI}
\frac{d^2}{dt^2}m_{{\rm CH}}(\Sigma(t))\Big|_{t=0}
&\leq&-\frac{|\Sigma_r|^{1/2}}{32\pi^{3/2}}\Big(\int_{\Sigma_r}(\Delta\phi)^2d\sigma_r-\int_{\Sigma_r}|\nabla\phi|^2d\sigma_r\Big).
\end{eqnarray}}
In the sequel, from the B${\rm \ddot{o}}$chner–Weitzenb${\rm \ddot{o}}$ck identity 
\begin{equation*}\label{bw}
-\int_{\Sigma_r}(\Delta\phi)^2d\sigma_r\leq-\frac{8\pi}{|\Sigma_r|}\int_{\Sigma_r}|\nabla\phi|^2d\sigma_r,
\end{equation*}
substituted into \eqref{auxACI} we arrive at
\begin{equation}\label{pi}
\frac{d^2}{dt^2}m_{{\rm CH}}(\Sigma(t))\Big|_{t=0}\leq -\frac{|\Sigma_r|^{1/2}}{32\pi^{3/2}}\Big(\frac{8\pi}{|\Sigma_r|}-1\Big)\int_{\Sigma_r}|\nabla\phi|^2d\sigma_r.
\end{equation}
Finally, using that \eqref{edo} we immediately conclude that $|\Sigma_r|<8\pi$ and hence, as $g_{\Sigma_r}=u(r)^2g_{\mathbb{S}^2}$, we have by Poincaré inequality
$$\frac{d^2}{dt^2}m_{{\rm CH}}(\Sigma(t))\Big|_{t=0}\leq -C(\Sigma_r)\int_{\Sigma_r}(\phi-\overline{\phi})^2d\sigma_r ,$$
where $C(\Sigma_r)$ is a positive constant. 
\end{proof}
Now we are in conditions to present the proof of Theorem \ref{gra}.
\subsection{Conclusion of the Proof of Theorem \ref{gra}}
\begin{proof}
Suppose that $\Sigma$ is a graph over a slice $\Sigma_0$
associated to a function $\varphi\in C^2(\Sigma_0)$. 
Following the same steps of \cite{mn} the operator $\mathcal{L}$ given by
{\setlength\arraycolsep{2pt}
\begin{eqnarray*}
\frac{d^2}{dt^2}m_{{\rm CH}}(\Sigma(t))\Big|_{t=0}&=&\mathcal{L}(\phi,\phi)
+\frac{2\pi^{1/2} Q^2}{|\Sigma_0|^{3/2}}\int_{\Sigma_0}|\nabla\phi|^2d\sigma_0
\end{eqnarray*}}
for $\phi\in C^2(\Sigma_0)$ with $\int_{\Sigma_0}\phi
d\sigma=0$ is strongly positive. Whence we obtain
\begin{equation}\label{mass}
m_{{\rm CH}}(\Sigma) -m_{{\rm CH}}(\Sigma_0) \geq\frac{1}{2}\langle\mathcal{L}\phi,\phi \rangle+  O(||\phi||_{C^2}||\phi||_{W^{2,2}}^2).
\end{equation}
Finally, it suffices to repeat exactly the same arguments used by Maximo and Nunes \cite{mn} in the last part of the proof of Theorem 1.2.
\end{proof}

\section{Stability result for charged Hawking mass}
The next result was motivated by Proposition 4.1 and Proposition 4.3 in \cite{mn} and describe, in a few words, an equality involving $\lambda_{1}(L)$ and the area of a two-sided strictly stable minimal surface which locally maximizes the charged Hawking mass on a three manifold endowed of a vector field $E.$ In addition, nice properties are obtained along such a surface.  
 
\begin{proposition}\label{locally}Let $(M^3, g)$ be a Riemannian manifold $E\in\mathfrak{X}(M)$ such that $div E=0$ with scalar curvature $R-2|E|^2\geq 2$. If $\Sigma\subset M$ is a two-sided compact strictly stable minimal surface
which locally maximizes the charged Hawking mass, then
\begin{equation*}
(\lambda_1(L)+1)|\Sigma|+\frac{16\pi^2Q(\Sigma)^2}{|\Sigma|}= 4\pi.
\end{equation*}
Moreover, along $\Sigma$, we have  $A = 0$, $R = 2|E|^2+2$,  $Ric(\nu, \nu) = -\lambda_1(L)$,  and its Gaussian
curvature $K_{\Sigma} = \frac{4\pi}{|\Sigma|}$
\end{proposition}
\begin{proof}
Applying Proposition \ref{2vari} joint with our assumption that $\Sigma$ locally maximizes the charged Hawking mass, we obtain
$$\lambda_1(L)^2\geq\Big(\frac{4\pi}{|\Sigma|}-\frac{16\pi^2Q(\Sigma)^2}{|\Sigma|^2}-1\Big)\lambda_1(L),$$
i.e.,
\begin{equation}
(\lambda_1(L)+1)|\Sigma|+\frac{16\pi^2Q(\Sigma)^2}{|\Sigma|}\geq 4\pi.
\end{equation}

On the other hand, choosing the function $\phi=1$  in the stability inequality \eqref{stability} and the Gauss equation 
$$R-2Ric(\nu,\nu) = 2K_{\Sigma} + |A|^2,$$
we conclude
\begin{equation}\label{auxinqR}
\lambda_1(L)|\Sigma| +\frac{1}{2}\int_{\Sigma}(R+|A|^2)d\sigma
\leq\int_{\Sigma}K_{\Sigma}d\sigma.
\end{equation}
Furthermore, since  $R-2|E|^2\geq 2$, we can use (\ref{auxinqR}) to get
\begin{eqnarray}\label{auxThm3}
(\lambda_1(L)+1)|\Sigma|&\leq& 4\pi-\int_{\Sigma}|E|^2d\sigma\nonumber\\
&\leq& 4\pi-\frac{16\pi^2Q(\Sigma)^2}{|\Sigma|},
\end{eqnarray}
where in the last inequality we have used H${\rm \ddot{o}}$lder inequality to deduce the reverse inequality.
Whence it follows the following equality
\begin{equation*}
(\lambda_1(L)+1)|\Sigma|+\frac{16\pi^2Q(\Sigma)^2}{|\Sigma|}= 4\pi.
\end{equation*}
\end{proof}

The next result was previously done in \cite{mn} and the authors recommend that reference for its prove. Such a result describe the construction of an one-parameter family $\Sigma(t)$ of the neighborhood of $\Sigma$ such that $\Sigma(t)$ is compact and has constant mean curvature.
\begin{proposition}\label{cmc} Let $(M^3, g)$ be a Riemannian manifold, $E\in\mathfrak{X}(M)$ such that $div E=0$ and scalar curvature satisfying
$R-2|E|^2\geq 2$. If $\Sigma\subset M$ is a stable compact embedded minimal surface satisfying
\begin{equation*}
(\lambda_1(L)+1)|\Sigma|+\frac{16\pi^2Q(\Sigma)^2}{|\Sigma|}= 4\pi.
\end{equation*}
then there exist $\epsilon > 0$ and a smooth function $\mu : \Sigma\times (-\epsilon,\epsilon)\rightarrow\mathbb{R}$ such that
$$\Sigma(t) = \{\exp_x(\mu(t, x)\nu(x)); x \in\Sigma\}$$
is a family of compact surfaces with constant mean curvature. Moreover, the
following properties hold
$$\mu(0, x) = 0,\ \ \ \frac{\partial\mu}{\partial t} (0, x) = 1\ \ \  and\ \ \ \int_{\Sigma}(\mu(t, \cdot)-t)d\sigma = 0$$
for each $x \in\Sigma$ and for each $t \in (-\epsilon, \epsilon)$.
\end{proposition}

With the aid of this proposition we define a mapping $f_t : \Sigma\rightarrow M$ by $f_t(x)=\exp_x(\mu(x, t)\nu(x))$ for each $t \in (-\epsilon, \epsilon)$. Note that $\Sigma_0=\Sigma$ for each $x\in\Sigma$. Let $\nu_t(x)$
be the unit vector field normal along $\Sigma(t)$ such that $\nu_0(x) = \nu(x)$ for all $x\in\Sigma$ and let us
denote $d\sigma_t$ the element of volume of $\Sigma(t)$ with respect to the induced metric by $f_t$.

Next, we consider the Jacobi operator 
$$L(t)=\Delta_{\Sigma(t)}+Ric(\nu_t,\nu_t)+|A_t|^2$$ 
where $\Delta_{\Sigma(t)}$ denotes the Laplacian of $\Sigma(t)$ in the induced metric and $A_t$ is the second fundamental form of $\Sigma(t)$ with respect to $\nu_t$.

Let $H(t)$ denote the mean curvature of $\Sigma(t)$ with respect to $\nu_t$ as well as the lapse function $\rho_t:\Sigma(t)\rightarrow\mathbb{R}$ which is defined by
$$\rho_t(x) =\Big\langle\nu_t(x),\frac{\partial}{\partial t}f_t(x)\Big\rangle
,$$
for each $t\in(-\epsilon,\epsilon)$.

Following the above terminology we now state  the next lemma which corresponds to Theorem 3.2 of \cite{Huisken}.
\begin{lemma}
The function $\rho_t(x)$ satisfies $H'(t) = L(t)\rho_t$.
\end{lemma}

As a consequence of Proposition \ref{cmc} we have that $\rho_0 = 1$. So, we may apply Proposition \ref{locally} and the last lemma to conclude that for $t\in(-\epsilon, \epsilon)$ the mean curvature $H(t)$ of this CMC foliation $\Sigma(t)\subset M$ in a neighborhood of $\Sigma$ satisfies
\begin{equation}\label{H'}
\frac{d}{dt}H(t)\Big|_{t=0}=-\lambda_1(L)<0.
\end{equation}
Thus we can decreasing $\epsilon$ if necessary, to conclude that $H(t) < 0$ for $t\in (0, \epsilon)$ and $H(t) > 0$ for $t\in (-\epsilon, 0)$, in other words, $H$ is a decreasing function on $t$.

In order to proceed it is crucial to recall the next lemma obtained  in [\cite{mn}, Lemma 5.3], that will be used in order to
deduce the monotonicity of the charged Hawking mass along the foliation $\Sigma(t)$.

\begin{lemma}\label{auxineq} 
Consider $(M,g)$, $\Sigma$ and $\{\Sigma(t)\}_{t\in(-\epsilon,\epsilon)}$ as in Proposition \ref{cmc}, then
{\setlength\arraycolsep{2pt}
\begin{eqnarray*}
\int_{\Sigma(t)}(Ric(\nu_t,\nu_t)+|A_t|^2)(\rho_t-\overline{\rho_t})d\sigma_t &\geq&\frac{\lambda_{1}(L_{\Sigma})}{\overline{\rho_t}}\int_{\Sigma(t)}(\rho_t-\overline{\rho_t})^{2}d\sigma_t,
\end{eqnarray*}}
where $\overline{\rho_t} =\frac{1}{|\Sigma(t)|}\int_{\Sigma(t)}\rho_td\sigma_t$.
\end{lemma}

\subsection{Proof of Theorem \ref{teoprincipal}}
\begin{proof}
Firstly, let $(M^3, g)$ be a Riemannian three-manifold that contains a two-sphere $\Sigma$ satisfying our assumptions. It follows from Propostion \ref{ld} that the Jacobi operator is given by $L=\Delta-\lambda_1(L)$, and from the Proposition \ref{cmc}, we can construct a constant mean curvature
foliation by two-spheres $\{\Sigma(t)\}_{|t|<\epsilon}$ around $\Sigma=\Sigma_0$. Using this special variation we can deduce a nice expression for $\frac{d}{dt}m_{{\rm CH}}(\Sigma(t))$. More precisely, since the mean curvature  $H_{t} $ is constant, we may use (\ref{cmh}) to infer  
\begin{eqnarray*}
\frac{d}{dt}m_{{\rm CH}}(\Sigma(t))&=&-H_{t}\overline{\rho}_{t}|\Sigma_{t}|\left[\frac{1}{8\pi^{1/2}|\Sigma_{t}|^{1/2}}-\frac{3H_{t}^{2}|\Sigma_{t}|^{1/2}}{128\pi^{3/2}}-\frac{\Lambda|\Sigma_{t}|^{1/2}}{4^{3}\pi^{3/2}}-\frac{\pi^{1/2}Q(\Sigma_{t})^{2}}{2|\Sigma_{t}|^{3/2}}\right]\\
&&-\frac{2H_{t}|\Sigma_{t}|^{1/2}}{64\pi^{3/2}}\int_{\Sigma_{t}}(Ric(\nu_{t},\nu_{t})+|A_{t}|^{2})\rho_{t} d\sigma_{t}\\
&=&-\frac{H_{t}\overline{\rho}_{t}|\Sigma_{t}|^{1/2}}{32\pi^{3/2}}\left[4\pi-\frac{3|\Sigma_{t}|H_{t}^{2}}{4}-\frac{\Lambda|\Sigma_{t}|}{2}-\frac{16\pi^{2}Q(\Sigma_{t})^{2}}{|\Sigma_{t}|}\right]\\
&&-\frac{H_{t}\overline{\rho}_{t}|\Sigma_{t}|^{1/2}}{32\pi^{3/2}}\int_{\Sigma_{t}}(Ric(\nu_{t},\nu_{t})+|A|^{2})d\sigma_{t}\\
&&-\frac{H_{t}|\Sigma_{t}|^{1/2}}{32\pi^{3/2}}\int_{\Sigma_{t}}(Ric(\nu_{t},\nu_{t})+|A|^{2})(\rho_{t}-\overline{\rho}_{t})d\sigma_{t}.
\end{eqnarray*}
Furthermore, this last expression can be rewritten using the Gauss equation as 
\begin{eqnarray*}
\frac{d}{dt}m_{{\rm CH}}(\Sigma(t))&=&-\frac{H_{t}\overline{\rho}_{t}|\Sigma_{t}|^{1/2}}{64\pi^{3/2}}\left[ \int_{\Sigma_{t}}(R-\Lambda-2|E|^{2})d\sigma_{t}+\int_{\Sigma_{t}}\left(|A_{t}|^{2}-\frac{1}{2}H_{t}^{2}\right)d\sigma_{t}\right]\\
&&-\frac{H_{t}\overline{\rho}_{t}|\Sigma_{t}|^{1/2}}{32\pi^{3/2}}\left[\int_{\Sigma_{t}}|E|^{2}d\sigma_{t}-\frac{16\pi^{2}Q(\Sigma_{t})^{2}}{|\Sigma_{t}|}\right]\\
&&-\frac{H_{t}|\Sigma_{t}|^{1/2}}{32\pi^{3/2}}\int_{\Sigma_{t}}(Ric(\nu_{t},\nu_{t})+|A_{t}|^{2})(\rho_{t}-\overline{\rho}_{t})d\sigma_{t}.
\end{eqnarray*}

This last identity together with Lemma~\ref{auxineq} and \eqref{H'}, where we have the function $H(t)$ decreasing, allow us to conclude that $\frac{d}{dt}m_{{\rm CH}}(\Sigma(t))\geq 0$ for $t\in [0,\epsilon )$ and $\frac{d}{dt}m_{{\rm CH}}(\Sigma(t))\leq 0$  for $t\in (-\epsilon,0]$. Therefore, we obtain
\begin{equation*}
m_{{\rm CH}}(\Sigma(t))\geq m_{{\rm CH}}(\Sigma),
\end{equation*}
for $t\in(-\epsilon,\epsilon)$. At the same time, $\Sigma$ locally maximizes the charge Hawking mass, i.e., $m_{{\rm CH}}(\Sigma(t))\leq m_{{\rm CH}}(\Sigma)$. Whence, it follows that $\frac{d}{dt}m_{{\rm CH}}(\Sigma(t))\equiv 0$ for $t\in(-\epsilon,\epsilon)$. Consequently, we have that $\Sigma(t)$ is umbilic, $R-2|E|^2 = 2$, $|E|$ is constant along $\Sigma(t)$ and $\rho_{t}\equiv\overline{\rho_{t}}$. In addition, as $\rho_{t}\equiv\overline{\rho_{t}}$, it is not difficult to see that $\mu(t,x)=t,\;\forall\;(t,x)\in(-\delta,\delta)\times\Sigma.$
Therefore, up to isometric, there is a small neighbourhood of $\Sigma$ such that the metric can be written as $g = dt^2 + g_{\Sigma_t}.$

Now, we are in position to apply \cite{Huisken} to conclude that the induced metric on $\Sigma(t)$ evolves as
$$\frac{\partial}{\partial t}g_{\Sigma(t)}=-2\rho_tA_t.$$
Since $\rho_t\equiv 1$, $\Sigma(t)$ is umbilic and $H(t)$ is constant we obtain
\begin{equation}
\frac{\partial}{\partial t}g_{\Sigma(t)}=-H(t)g_{\Sigma(t)},
\end{equation}
for all $t\in(-\epsilon,\epsilon)$.

From this, it follows that $g_{\Sigma(t)} = u_{Q,a}(t)^2g_{\mathbb{S}^2}$ for all $t\in(-\epsilon,\epsilon)$, where $u_{Q,a}(t) = ae^{-\frac{1}{2}\int_0^tH(s)ds}$ with $a^2 =\frac{|\Sigma|}{4\pi}$ and $Q=Q(\Sigma)=a^2|E|$.

It is not difficult to show that the function $u(t)$ solves the equation
\begin{equation}\label{edo1}
u''(t)=\frac{1}{2}\Big(\frac{1-u'(t)^2}{u(t)}\Big)-\frac{1}{2}\Big(\frac{u(t)^4+Q^2}{u(t)^3}\Big).
\end{equation}

Finally, by uniqueness of solutions of the ODE, it follows that $((-\epsilon,\epsilon)\times\Sigma,g) $ is isometric to a piece of the Reissner-Nordstr${\rm \ddot{o}}$m-de Sitter space with positive mass and we can conclude that the induced metric by $f(t, x)=\exp_x(t\nu(x))$ on $(-\epsilon,\epsilon)\times\Sigma$ is given by $dt^2 + u_{Q,a}(t)^2g_{\mathbb{S}^2}$.
\end{proof}

\section{Proof of Theorem~\ref{teoprincipalN}}
\subsection{Proof of Theorem \ref{teoprincipalN}}
\begin{proof}
The computation to obtain inequality (\ref{inqN}) was done previously in Proposition~\ref{locally}. In fact, considering the stable case, inequalities (\ref{auxinqR}) and (\ref{auxThm3}) will still be true, now without $\lambda_{1}(L),$ so we have immediately the desire inequality and, if equality occurs in (\ref{inqN}), then the following relations holds on $\Sigma$: 
$$A\equiv0,\;R=2+2|E|^{2}, K_{\Sigma}=1+|E|^{2}\; \text{and}\;Ric(\nu,\nu)=0,$$
where $|E|$ is constant. As a consequence of those remarks, we may use proposition~\ref{cmc} in order to deduce some properties on $\Sigma(t),$ which is the family of compact surfaces with constant mean curvature in a neighbourhood  of $\Sigma.$ To begin, note that $H'(t)=-\Delta_{t}\rho_{t}-(Ric(\nu_{t},\nu_{t})+|A_{t}|^{2})\rho_{t}.$ Because $\rho_{0}(x)=1$ we have from Gauss equation 
\begin{eqnarray*}
H'(t)\frac{1}{\rho_{t}}&=&-\frac{1}{\rho_{t}}\Delta_{t}\rho_{t}-\left(\frac{R_{t}}{2}-K_{\Sigma_{t}}+\frac{H_{t}^{2}}{2}-\frac{|A_{t}|^{2}}{2}\right)\\
&\leq& -\frac{1}{\rho_{t}}\Delta_{t}\rho_{t}-(|E_{t}|^{2}+1)+K_{\Sigma_{t}}.
\end{eqnarray*}
On integrating the last relation we obtain
\begin{eqnarray*}
H'(t)\int_{\Sigma_{t}}\frac{1}{\rho_{t}}d\sigma_{t}&\leq& -\int_{\Sigma_{t}}\frac{1}{\rho_{t}^{2}}|\nabla_{t}\rho_{t}|^{2}d\sigma_{t}-\int_{\Sigma_{t}}|E_{t}|^{2}d\sigma_{t}-|\Sigma_{t}|+4\pi\\
&\leq&-16\pi^{2}\frac{Q(\Sigma_{t})^{2}}{|\Sigma_{t}|}-|\Sigma_{t}|+4\pi.
\end{eqnarray*}
Now, remember that ${\rm div}(E)=0$ as well as $Q(\Sigma_{t})=Q(\Sigma_{0})=Q_{0},$ $\forall\; t\in(-\epsilon,\epsilon).$ This jointly with equality in (\ref{inqN}) yields
\begin{eqnarray*}
H'(t)
\int_{\Sigma_{t}}\dfrac{1}{\rho_{t}}d\sigma_{t}
&\leq &-16\pi^{2}\dfrac{Q(\Sigma_{0})^{2}}{|\Sigma_{t}|}-|\Sigma_{t}|+16\pi^{2}\dfrac{Q(\Sigma_{0})^{2}}{|\Sigma_{0}|}+|\Sigma_{0}|\\
&\leq &16\pi^{2}Q_{0}^{2}\Big(\dfrac{1}{|\Sigma_{0}|}-\dfrac{1}{|\Sigma_{t}|}\Big),
\end{eqnarray*}
since $|\Sigma_{0}|\leq|\Sigma_{t}|,\forall\, t\in(-\varepsilon,\varepsilon).$ 
Therefore we infer 
\begin{eqnarray*}
|\Sigma_{t}| H'(t)\int_{\Sigma_{t}}\dfrac{1}{\rho_{t}}d\sigma_{t}
&\leq &16\pi^{2}Q_{0}^{2}\dfrac{1}{|\Sigma_{0}|}\Big(|\Sigma_{t}|-|\Sigma_{0}|\Big)\\
&=&16\pi^{2}Q_{0}^{2}\dfrac{1}{|\Sigma_{0}|}\int_{0}^{t}\frac{d}{ds}|\Sigma(s)|ds\\
&=&16\pi^{2}Q_{0}^{2}\dfrac{1}{|\Sigma_{0}|}\int_{0}^{t}\left(\int_{\Sigma_{s}}H(s)\rho_{s}d\sigma_{s}\right)ds,
\end{eqnarray*}
since $H(s)$ is constant on $\Sigma_{s}$ we obtain
\begin{eqnarray}\label{hprime}
|\Sigma_{t}| H'(t)\int_{\Sigma_{t}}\dfrac{1}{\rho_{t}}d\sigma_{t}&\leq &16\pi^{2}Q_{0}^{2}\dfrac{1}{|\Sigma_{0}|}\int_{0}^{t}\left(H(s)\int_{\Sigma_{s}}\rho_{s}d\sigma_{s}\right)ds.
\end{eqnarray}

In order to simplify some computations we set $\phi(t)=\int_{\Sigma_{t}}\frac{1}{\rho_{t}}d\sigma_{t},$ $\xi=\int_{\Sigma_{t}}\rho_{t}d\sigma_{t}$ and $\kappa=16\pi^{2}Q_{0}^{2}\frac{1}{|\Sigma_{0}|}$. Then, we can write the last inequality as 
$$|\Sigma_{t}|H'(t)\phi(t)\leq\kappa\int_{0}^{t}H(s)\xi(s)ds.$$

On the other hand, taking into account that $\rho_{0}=1,$ we can assume $$\frac{1}{2}<\rho_{t}(x)<2,$$ 
$\forall\;t\in(-\epsilon,\epsilon)$. Whence we derive that $\frac{1}{2}|\Sigma_{t}|<\xi(t)<2|\Sigma_{t}|,$ $\forall\;t\in(-\epsilon,\epsilon)$. To proceed, shrinking $\epsilon$ if necessary, we can suppose $\frac{1}{2}|\Sigma_{0}|<|\Sigma_{t}|<2|\Sigma_{0}|$, which implies that 
$$\frac{1}{4}|\Sigma_{0}|<\xi(t)<4|\Sigma_{0}|.$$
Similarly, it is easy to check that  $\frac{1}{4}|\Sigma_{0}|<\phi(t)<4|\Sigma_{0}|.$ Whence we deduce
\begin{equation}\label{auxEst}
\frac{1}{\phi(t)}<\frac{4}{|\Sigma_{0}|}\;\;\;\text{and}\;\;\;\xi<4|\Sigma_{0}|.
\end{equation}

\textbf{Claim 1:} $\exists\; \delta<\epsilon$ such that $H(t)\leq0,$ $\forall\;t\in[0,\delta).$

In fact, taking an arbitrary $\delta<\epsilon,$ suppose by contradiction that there exists $t_{0}\in[0,\delta)$ such that $H(t_{0})>0.$ At the same time, we consider the set $I=\{t\in[0,t_{0}];\;H(t)\geq H(t_{0})\}$ and denote, for simplicity, its infimum by $t^{\ast}=inf I.$ Notice that if we have $t^{\ast}>0,$ then by the Mean Value Theorem, there exists $t_{1}\in(0,t^{\ast})$ such that $H(t^{\ast})=H'(t_{1})t^{\ast}$ (it is important to remember that $H(0)=0$) and consequently, we may use (\ref{hprime}) and (\ref{auxEst}) to deduce the following estimate
\begin{eqnarray*}
H(t^{\ast})&\leq&\frac{t^{\ast}\kappa}{|\Sigma_{t_{1}}|\phi(t_{1})}\int_{0}^{t_{1}}H(s)\xi(s)ds\\
&<&\frac{8\kappa t^{\ast} H(t^{\ast})}{|\Sigma_{0}|^{2}}\int_{0}^{t_{1}}\xi(s)ds\\
&<&\frac{32\kappa}{|\Sigma_{0}|} (t^{\ast})^{2} H(t^{\ast}),
\end{eqnarray*}
which is a contradiction if we take $\delta<\sqrt{\frac{|\Sigma_{0}|}{32\kappa}}.$ Hence, we obtain $infI=0.$ Finally, since we already know that $infI=0$, we get
$$0=H(0)\geq H(t_{0})>0,$$ 
providing another contradiction. Whence, our claim is true.
 
Thus, it is immediate to deduce $\frac{d}{dt}|\Sigma_{t}|\leq0,$ which implies that $|\Sigma_{t}|\leq|\Sigma_{0}|,\;\forall\;t\in[0,\delta).$ Similarly, we have $|\Sigma_{t}|\leq0,\;\forall\;t\in(-\delta,0].$ Since $\Sigma$ is area minimizing, we obtain $|\Sigma_{t}|=|\Sigma_{0}|,$ $\forall\;t\in(-\delta,\delta).$ In addition, each level $\Sigma_{t}$ is totally geodesic, $|E_{t}|$ is constant,  $K_{\Sigma_{t}}=|E_{t}|^{2}+1,$ $R=2|E_{t}|^{2}+2$ and $Ric(\nu_{t},\nu_{t})=0.$  

Now, it follows from  standard argument that there exists a neighbourhood $U\subset M^{3}$ of $\Sigma$ such that $(U,g)$ is isometric to $((-\delta,\delta)\times\Sigma,dt^{2}+g_{\Sigma}),$ which finishes the proof of the theorem.
\end{proof}

\section{Appendix}
In this appendix we will obtain an area-charge inequality which was motivated by Cruz, Lima e de Sousa in \cite{brad}, see section 7 of \cite{brad} for more details.

For what follows, consider a Riemannian 3-manifold $(M^3,g)$,  $ E \in\mathfrak{X}(M)$ and $V\in C^{\infty}(M)$ such that $V > 0$. The Einstein-Maxwell equations with cosmological constant $\Lambda$ for the electrostatic space-time associated to $(M, g, V, E)$ is the following system of equations

\begin{equation}\label{elst}
\left\{
\begin{array}{lcl}
Hess_g V& =& V (Ric_g -\Lambda g+2E^{\flat}\otimes E^{\flat}-|E|^2g);\\
\Delta_gV& = &(|E|^2-\Lambda)V;\\
div_g E& =& 0;\ \  curl_g(V E) = 0,
\end{array}
\right.  
\end{equation}
where $E^{\flat}$ is the one-form metrically dual to $E$.

This motivates the following definition.
\begin{definition}
We say $(M, g, V, E)$ is an electrostatic system if $(M, g)$ is a Riemannian manifold, $V \in C^{\infty}(M)$ and is not identically zero, $E \in\mathfrak{X}(M)$,
and system \eqref{elst} is satisfied for some constant $\Lambda\in\mathbb{R}$. Moreover, the system is complete, if $(M, g)$ is complete.
\end{definition}

There are two important examples of electrostatic systems that we would like to report here. The first one is the Riemannian manifold 
$$\left([r_{+},r_{c}]\times\mathbb{S}^{2},g_{m,Q,\Lambda}=\Big(1-\frac{\Lambda}{3}r^2+\frac{Q^2}{r^2}-\frac{2m}{r}\Big)^{-1}dr^2+r^2g_{\mathbb{S}^2}\right),$$
where the potential function and electric field are given, respectively, by $V=\sqrt{1-\frac{\Lambda}{3}r^2+\frac{Q^2}{r^2}-\frac{2m}{r}}$ and $E=\frac{Q}{r^{2}}V(r)\partial_{r}.$ Here, $r_{+}<r_{c}$ are positive zeros of $V$ as described in Section~\ref{sectionInt} - Introduction. The next one is the standard cylinder over $\mathbb{S}^{2}$ with product metric, 
$$\left(\left[0,\frac{\pi}{\alpha}\right]\times\mathbb{S}^{2},ds^{2}+\rho^{2}g_{\mathbb{S}^{2}}\right),$$
where $\alpha=\sqrt{\frac{\Lambda\rho^{4}-Q^2}{\rho^{4}}}$  and $\rho\in(\frac{1}{\sqrt{2\Lambda}},\frac{1}{\sqrt{\Lambda}})$ is a positive double root of Eq. (\ref{defrn}). In this case the potential function and electric field are given by $V(s)=sin(\alpha s)$  and $E(s)=\frac{Q}{\rho^2}\partial_{s},$ respectively.

In the sequel, we consider $MAX(V)$ to be the set where the maximum of $V$ is achieved, namely, $$MAX(V)=\{p\in M;\,V(p)=V_{max}\}$$ and let $N$ be a single connected component of $M\setminus MAX(V).$ With these settings we will prove the following result.

\begin{theorem}\label{ACineq}
Let ($M^3, g, V,E)$ be a compact electrostatic system, such that $V^{-1}(0) = \partial M$. Assume $\sup_{M} |E|^2 \leq \Lambda$. Let $N$ be a single connected component of $M\setminus MAX(V)$, and let $\partial N=\partial M\cap N$ be a the non-empty and possibly disconnected. Write  $\partial N=\displaystyle{\cup_{i=1}^{k}\partial N_{i}}.$ Then, there is a boundary component of $\partial N$ which is diffeomorphic to a two-sphere and
$$\sum_{i=1}^{k}k_i\Big(\Lambda |\partial N_i|+\frac{48\pi^2Q(\partial N_{i})^2}{|\partial N_i|}\Big)\leq 6\pi \sum_{i=1}^{k}k_i\chi(\partial N_i),$$
where $k_i$ is the restriction of $|\nabla_gV|$ to $\partial N_i$. Moreover, the equality holds if and only if $E\equiv 0,$ and $(M^3, g)$ is isometric to the de Sitter solution.
\end{theorem}

Finally, in considering that $\partial N$ is connected we immediately obtain the following result which can be compared with Theorem 25 in \cite{brad}.

\begin{corollary}\label{inq1B}
Let ($M^3, g, V,E)$ be a compact electrostatic system, such that $V^{-1}(0) = \partial M$. Assume that $\sup_M |E|^2 \leq \Lambda$. Let $N$ be a single connected component of $M\setminus MAX(V)$ and suppose $\partial N=\partial M\cup N$ is connected. Then $\partial N$ is is diffeomorphic to a two-sphere and 
$$\Lambda|\partial N|+\frac{48\pi^2Q(\partial N)^2}{|\partial N|}\leq 12\pi$$
where $Q(\partial N)$ is the charge relative to $\partial N$. Moreover, the equality holds if, and only if, $E\equiv 0$ and $(M^3, g)$ is isometric to the de Sitter solution.
\end{corollary}

For our purposes we will provide a Robinson-Shen type identity that plays a crucial role in this section. It was essentially motivated by \cite{Rob} and \cite{Shen97}.

\begin{lemma}\label{Shen} (Robinson-Shen Type Identity)
Let $(M^{n},\,g, V, E)$ be a compact electrostatic system. 
Then we have:
\begin{eqnarray*}
{\rm div}\left[\frac{1}{V}\Big(\nabla|\nabla_gV|^{2}-\frac{2\Delta_g V}{n}\nabla_gV\Big)\right]=\frac{2}{V}|\mathring{Hess_g V}|^{2}+\frac{2(n-1)}{n}\langle\nabla |E|^2,\nabla_gV\rangle,
\end{eqnarray*} where $\mathring{Hess\,V}=Hess\,V-\frac{\Delta_g V}{n}g.$
\end{lemma}
\begin{proof}
First of all, we will consider the field $X$ in the interior of $M^3$ given by $X=\frac{1}{V}\Big(\nabla|\nabla_gV|^{2}-\frac{2\Delta_g V}{n}\nabla_gV\Big).$ Whence, a straightforward computation yields
{\setlength\arraycolsep{2pt}
\begin{eqnarray*}
\frac{V}{2}{\rm div}\,X&=&\frac{1}{2}\Delta|\nabla_gV|^{2}+\frac{V}{2}\langle\nabla_gV^{-1},\nabla|\nabla_gV|^{2}\rangle-\frac{1}{n}(\Delta_g V)^{2}-\frac{V}{n}\langle\nabla\left(\frac{\Delta_g V}{V}\right),\nabla_gV\rangle\\
&=&\frac{1}{2}\Delta|\nabla_gV|^{2}-\frac{1}{V}Hess_g V(\nabla_gV,\nabla_gV)-\frac{1}{n}(\Delta_g V)^{2}\nonumber\\&&-\frac{1}{n}\langle\nabla\Delta_g V,\nabla_gV\rangle+\frac{\Delta_g V}{nV}|\nabla_gV|^{2},
\end{eqnarray*}} which can be rewritten, using the classical Bochner formula, as follows
{\setlength\arraycolsep{2pt}
\begin{eqnarray*}
\frac{V}{2}{\rm div}X&=&|\mathring{Hess_g V}|^{2}+Ric(\nabla_gV,\nabla_gV)-\frac{1}{V}Hess_g V(\nabla_gV,\nabla_gV)\\
&&+\frac{n-1}{n}\langle\nabla\Delta_g V,\nabla_gV\rangle+\frac{\Delta_g V}{nV}|\nabla_gV|^{2}.
\end{eqnarray*}}
Now note that the assumption $curl_g(VE)=0$ implies that $E$ and $\nabla_gV$ are linearly dependent and consequently  $\langle E,\nabla_gV\rangle^2=|E|^2|\nabla_gV|^2$. Thus, it suffices to use system \eqref{elst} to obtain
{\setlength\arraycolsep{2pt}
\begin{eqnarray*}
\frac{V}{2}{\rm div}X&=&|\mathring{Hess_g V}|^{2}+\frac{n-1}{n}\langle\nabla\Delta_g V,\nabla_gV\rangle-\frac{n-1}{n}\frac{\Delta_g V}{V}|\nabla_gV|^{2}\\
&=&|\mathring{Hessf}|^{2}+\frac{n-1}{n}V\langle\nabla |E|^2,\nabla_gV\rangle.
\end{eqnarray*}} This completes the proof of the lemma.
\end{proof}

\begin{proposition}
\label{prop1a}
Let ($M^3, g, V,E)$ be a compact electrostatic system, such that $V^{-1}(0) = \partial M$ and $\sup_M |E|^2 \leq \Lambda$. Let $N$ be a connected component of $M\setminus MAX(V).$ Then $N\cap \partial M\neq \emptyset.$
\end{proposition}

\subsection{Proof of Theorem~\ref{ACineq}}

\begin{proof} Firstly, we already known by Lemma \ref{Shen} that

\begin{equation}\label{eq4a}
{\rm div}\left[\frac{1}{V}\Big(\nabla|\nabla_gV|^{2}-\frac{2\Delta_g V}{n}\nabla_gV\Big)\right]=\frac{2}{V}|\mathring{Hess_g V}|^{2}+\frac{4}{3}\langle\nabla |E|^2,\nabla_gV\rangle,\end{equation} Next, taking into account that $M^n$ is compact and using the properties of the function $V,$ it follows that there exists a $\epsilon>0$ such that the set $\{V=t\}$ is regular for every $0\le t\le \epsilon$ and $V_{max}-\epsilon\le t<V_{max}.$ 

By one hand, write $X=\frac{1}{V}\Big(\nabla|\nabla_gV|^{2}-\frac{2\Delta_g V}{n}\nabla_gV\Big)$ upon integrating \eqref{eq4a} over $\{\varepsilon<V<V_{max}-\varepsilon\}\cap N$ we obtain
{\setlength\arraycolsep{2pt}
\begin{eqnarray}\nonumber
\int_{\{V=V_{max}-\varepsilon\}\cap N}\left\langle X,\,\nu\right\rangle d\sigma &\ge&\int_{\{V=\varepsilon\}\cap N}\left\langle X,\,\nu\right\rangle d\sigma-\frac{4}{3}\int_{\{\varepsilon<V<V_{max}-\varepsilon\}\cap N}|E|^2\Delta_g V dv\\\label{eq4b}
&&+\frac{4}{3}\int_{\{V=V_{max}-\varepsilon\}\cap N}|E|^2\left\langle \nabla_gV,\,\nu\right\rangle d\sigma-\frac{4}{3}\int_{\{V=\varepsilon\}\cap N}|E|^2\left\langle \nabla_gV,\,\nu\right\rangle d\sigma 
\end{eqnarray}}
where $\nu=\frac{\nabla_gV}{|\nabla_gV|}$ is the unit normal to $\{\varepsilon<V<V_{max}-\varepsilon\}\cap N$.

On the other hand, by using system \eqref{elst} as well as the fact that $\langle E,\nabla_gV\rangle^2=|E|^2|\nabla_gV|^2$ we obtain
\begin{equation*}
\Big\langle\frac{1}{V}\Big(\nabla|\nabla_gV|^{2}-\frac{2\Delta_g V}{n}\nabla_gV\Big),\frac{\nabla_gV}{|\nabla_gV|}\Big\rangle=2|\nabla_gV|Ric(\nu,\nu)+\frac{4}{3}(|E|^2-\Lambda)|\nabla_gV|.
\end{equation*} Plugging this data into \eqref{eq4b} we deduce
{\setlength\arraycolsep{2pt}
\begin{eqnarray}\nonumber
\int_{\{V=V_{max}-\varepsilon\}\cap N}|\nabla_gV|\Big(Ric(\nu,\,\nu)-\frac{2}{3}\Lambda\Big)d\sigma&\geq& \int_{\{V=\varepsilon\}\cap N}|\nabla_gV|\Big(Ric(\nu,\nu)-\frac{2}{3}\Lambda\Big)d\sigma\\\nonumber
&&+\frac{2}{3}\int_{\{\varepsilon<V<V_{max}-\varepsilon\}\cap N}(\Lambda-|E|^2)|E|^2V dv\\\label{eq4c}
&&\geq \int_{\{V=\varepsilon\}\cap N}|\nabla_gV|\Big(Ric(\nu,\nu)-\frac{2}{3}\Lambda\Big)d\sigma,
\end{eqnarray}}
where in the last inequality we use that $\sup_M|E|^2<\Lambda$.

Now, we claim that 
\begin{equation}
\label{eq4d}
\liminf_{\varepsilon \to 0}\int_{\{V=V_{max}-\varepsilon\}\cap N}|\nabla_gV|\Big(Ric(\nu, \nu)-\frac{2}{3}\Lambda\Big)d\sigma=0.
\end{equation} To prove this, since $Ric(\nu, \nu)-\frac{2}{3}\Lambda$ is bounded, it suffices to show that
$$\liminf_{t\to V_{max}}\int_{\{V=t\}\cap N}|\nabla_gV|d\sigma=0.$$ Indeed, by  Theorem 2.2 (Reverse Łojasiewicz Inequality) in \cite{BCM}, given a $\theta<1,$ there exists a neighbourhood  $\mathcal{U}$ of $MAX(V)$ and a real number $C>0$ such that
$$|\nabla_gV|^{2}(x)\leq C\left(V_{\max}-V(x)\right)^{\theta},$$ for every $x\in\mathcal{U}.$ 

Proceeding, for $t$ close to $V_{\max}$ we immediately have

$$\int_{\{V=t\}\cap N}|\nabla_gV|\leq C^{\frac{1}{2}}\left(V_{\max}-t\right)^{\frac{\theta}{2}}area\big(\{V=t\}\cap N\big).$$  From this, it follows that 
$$\liminf_{t\to V_{max}}\int_{\{V=t\}\cap E}|\nabla V|d\sigma=0,$$ as wished.

Now, taking $\liminf\limits_{\varepsilon \to 0}$ in (\ref{eq4c}) we arrive at

\begin{equation*}
\int_{\partial N}|\nabla_gV|\Big(Ric(\nu,\nu)-\frac{2}{3}\Lambda\Big)d\sigma\leq 0.
\end{equation*} At the same time, by  Gauss equation, Gauss-Bonnet Theorem, Cauchy-Schwarz inequality and H${\rm \ddot{o}}$lder’s inequality we have that
{\setlength\arraycolsep{2pt}
\begin{eqnarray*}
0&\ge & \int_{\partial N}|\nabla_gV|\Big(|E|^2+\frac{\Lambda}{3}-\frac{R^{\partial N}}{2}\Big)d\sigma\\
&\ge &\sum_{i=1}^kk_i\Big[\frac{1}{|\partial N_i|}\Big(\int_{\partial N_i}\langle E,N \rangle\Big)^2+|\partial N_i|\frac{\Lambda}{3}-2\pi\chi(\partial N_i)\Big]\\
&=& \sum_{i=1}^kk_i\Big(\frac{16\pi^2Q(\partial N_i)^2}{|\partial N_i|}+|\partial N_i|\frac{\Lambda}{3}\Big)-2\pi\sum_{i=1}^kk_i\chi(\partial N_i)
\end{eqnarray*}} 
as asserted.

Finally, if the equality holds it suffices to use \eqref{eq4c} jointly with the fact that $\sup_M|E|^2<\Lambda$ to conclude that $|E|\equiv 0$ on $N$. Therefore, we may apply  Theorem 2.4 in \cite{BM1} to conclude that $M^3$ is isometric to the standard hemisphere, which completes the proof of the theorem.
\end{proof}

\section*{Acknowledgements}
The third author would like to thank the Department of Mathematics of UFC for its support,
where part of the work was started. He would like to extend his special thank to Professor
C\'icero Tiarlos Cruz for very helpful conversations. The first author was partially supported by PPP/FAPEPI/MCT/CNPq, Brazil [Grant: 007/2018] and CNPq/Brazil [Grant: 422900/2021-4], while the second one is partially supported by CNPq/Brazil [Grant:313407/2020-7].

\bibliographystyle{amsplain}

\end{document}